\documentclass{amsart}
\usepackage{amsmath,amsthm}
\usepackage{amsfonts,amssymb}
\usepackage{enumerate}
\usepackage{accents,color}
\usepackage{graphicx}

\newtheorem{theorem}{Theorem}[section]
\newtheorem{lemma}[theorem]{Lemma}
\newtheorem{proposition}[theorem]{Proposition}
\newtheorem{corollary}[theorem]{Corollary}

\theoremstyle{definition}

\theoremstyle{remark}

\numberwithin{equation}{section}

\def \la {\lambda}
\def \al {\alpha}

\def\A{{\mathbf A}}
\def\B{{\mathbf B}}
\def\C{{\mathbf C}}
\def\D{{\mathbf D}}
\def\E{{\mathbf E}}

\def\G{{\mathbf G}}

\def\HH{{\mathcal H}}
\def\PP{{\mathcal P}}

\def\d{\delta}

\begin{document}

\title[Hardy's inequalities in finite dimensional Hilbert spaces]{Hardy's inequalities in finite dimensional Hilbert spaces}
\thanks{Research supported by the Brazilian Science Foundations FAPESP under Grants 2016/09906-0 and  2016/10357-1 and CNPq under Grant 306136/2017-1 and the 
Bulgarian National Research Fund through Contract DN 02/14.}

\author{Dimitar K. Dimitrov}
\address{Departamento de Matem\'atica, IBILCE, Universidade Estadual Paulista, 15054-000
S\~{a}o Jos\'{e} do Rio Preto, SP, Brazil }

\email{d\_k\_dimitrov@yahoo.com}

\author{Ivan Gadjev}
\address{Faculty of Mathematics and Informatics, Sofia University "St. Kliment Ohridski", 5 James Bourchier Blvd.,
1164 Sofia, Bulgaria}
\email{gadjev@fmi.uni-sofia.bg}

\author{Geno Nikolov}
\address{Faculty of Mathematics and Informatics, Sofia University "St. Kliment Ohridski", 5 James Bourchier Blvd.,
1164 Sofia, Bulgaria}
\email{geno@fmi.uni-sofia.bg}

\author{Rumen Uluchev}
\address{Faculty of Mathematics and Informatics, Sofia University "St. Kliment Ohridski", 5 James Bourchier Blvd.,
1164 Sofia, Bulgaria}
\email{rumenu@fmi.uni-sofia.bg}



\dedicatory{}


\begin{abstract}
We study the behaviour of the smallest possible constants $d_n$ and $c_n$  in Hardy's inequalities 
$$
\sum_{k=1}^{n}\Big(\frac{1}{k}\sum_{j=1}^{k}a_j\Big)^2\leq
d_n\,\sum_{k=1}^{n}a_k^2, \qquad  (a_1,\ldots,a_n) \in \mathbb{R}^n
$$
and 
$$
\int_{0}^{\infty}\Bigg(\frac{1}{x}\int\limits_{0}^{x}f(t)\,dt\Bigg)^2
dx \leq c_n \int_{0}^{\infty} f^2(x)\,dx, \ \ f\in \mathcal{H}_n,
$$
for the finite dimensional spaces $\mathbb{R}^n$ and $\mathcal{H}_n:=\{f\,:\,
\int_0^x f(t) dt =e^{-x/2}\,p(x)\ :\ p\in \mathcal{P}_n, p(0)=0\}$, where $\mathcal{P}_n$ is the set of real-valued algebraic polynomials of
degree not exceeding $n$. The constants $d_n$ and $c_n$ are identified as the smallest 
eigenvalues of certain Jacobi matrices and the two-sided estimates for $d_n$ and $c_n$ of the form
$$
4-\frac{c}{\ln n}< d_n, c_n<4-\frac{c}{\ln^2 n}\,,\qquad c>0\,
$$
are established. 
\end{abstract}

\maketitle

\section{Introduction and statement of the results}

Denote by $\ell^p_+$  and $L^p_+[0,\infty)$  the classes of nonnegative sequences $\{a_k\} \in \ell^p$ and of nonnegative functions $f \in L^p[0,\infty)$, respectively. 
In a series of papers Hardy \cite {H1919, H1920, H1925} proved the inequalities 
\begin{equation}\label{e1.1.5}
\sum_{k=1}^{\infty}\Big(\frac{1}{k}\sum_{j=1}^{k}a_j\Big)^p\leq
\Big(\frac{p}{p-1}\Big)^{p}\,\sum_{k=1}^{\infty}{a_k^p},\ \ \ p>1,\ \ \  \{ a_k \}\in \ell^p_+,
\end{equation}
and
\begin{equation}\label{e1.1}
\int\limits_{0}^{\infty}\Bigg(\frac{1}{x}
\int\limits_{0}^{x}f(t)\,dt\Bigg)^{p}dx\leq
\Big(\frac{p}{p-1}\Big)^{p}\int\limits_{0}^{\infty}\ [f(x)]^{p}\,dx,\ \ \ p>1,\ \ \  f\in L^p_+[0,\infty).
\end{equation}
Nowadays (\ref{e1.1.5}) and (\ref{e1.1}) are called Hardy's inequalities. We refer to the monographs \cite{HLP1967, KP2003, KMP2007} and
the references therein, as well as to the nice survey \cite{KMP2006}, for the history and 
various generalisations and applications of Hardy's inequalities. 

Despite that, under the assumptions that $a_k\geq 0$ and $f(x)\geq 0$, the only extremizers in \eqref{e1.1.5} and \eqref{e1.1}  are 
the sequence and the function which vanish identically (see \cite[p. 246]{HLP1967}), E. Landau \cite{Lan} observed that the constant $(p/(p-1))^{p}$ is the smallest possible one, for which 
\eqref{e1.1.5} and  \eqref{e1.1} hold. 



In the case $p=2$ the assumption for nonnegativity of $\,\{a_k\}\,$  and $f$ 
can be dropped, and inequalities \eqref{e1.1.5}  and \eqref{e1.1} become 
\begin{equation}\label{e1.2.5}
\sum_{k=1}^{\infty}\Big(\frac{1}{k}\sum_{j=1}^{k}a_j\Big)^2\leq
4\,\sum_{k=1}^{\infty}{a_k^2}\, 
\end{equation}
and
\begin{equation}\label{e1.2}
\int\limits_{0}^{\infty}\Bigg(\frac{1}{x}
\int\limits_{0}^{x}f(t)\,dt\Bigg)^{2}dx\leq
4\,\int\limits_{0}^{\infty}f^{2}(x)\,dx\,,
\end{equation}

Motivated by the question about the existence of ``almost extremizing'' sequences and function for the latter inequalities, 
in the present paper we examine the best constant $d_n$ and $c_n$ in the Hardy inequalities
\begin{equation}\label{e1.5}
\sum_{k=1}^{n}\Big(\frac{1}{k}\sum_{j=1}^{k}a_j\Big)^2\leq
d_n\,\sum_{k=1}^{n}{a_k^2}, \qquad (a_1,a_2,\ldots,a_n)\in
\mathbb{R}^n
\end{equation}
and 
\begin{equation}\label{e1.3}
\int\limits_{0}^{\infty}\Bigg(\frac{1}{x}\int\limits_{0}^{x}f(t)\,dt\Bigg)^2dx
\leq c_n \int\limits_{0}^{\infty}f^2(x)\,dx,\qquad f\in \HH_n,
\end{equation}
for the finite dimensional spaces $\mathbb{R}^n$ and $\HH_n$, where the latter is defined by 
$$
\HH_n=\{f\;:\; \int_0^x f(t)\, dt = e^{-x/2} p(x),\ p\in \PP_n,\ p(0)=0 \}\,.
$$

Since $\mathbb{R}^n \subset \ell^2$ and $\HH_n\subset L^2(0,\infty)$, it follows from \eqref{e1.2.5} and \eqref{e1.2}
that $d_n\leq 4$, $c_n\leq 4$.
We show that both $d_n$ and $c_n$ converge to $4$ as $n\to \infty$, although the convergence speed turns out to be rather slow.
Our main results read as follows: 

\begin{theorem}\label{t1.2}
Let $\,d_n\,$ be the smallest possible constant such that the
inequality \eqref{e1.5} holds. Then $\,d_n\,$ obeys the estimates
$$
4\Bigg(1-\frac{4}{\ln n +4}\Bigg)\leq d_n \leq
4\Bigg(1-\frac{8}{(\ln n + 4)^2}\Bigg)\,,\qquad n\geq 3.
$$
\end{theorem}

\begin{theorem}\label{t1.1}
The best constant $c_n$ in the Hardy inequality \eqref{e1.3} admits
the estimates
$$
4\Bigg(1-\frac{2}{\ln\frac{n+1}{2}+2}\Bigg)\leq c_n\leq
4\Bigg(1-\frac{8}{(\ln \frac{n+1}{2}+ 4)^2}\Bigg)\,.
$$
\end{theorem}

The natural question about the precise asymptotic behaviour of the second term in the above results, in terms of $n$, arises.  

We show that $d_n$ and $c_n$ coincide with the smallest eigenvalues of certain Jacobi matrices and then 
estimate them properly via lower bounds for the smallest zeros of the orthogonal polynomials associates with these matrices.
More precisely, we study the smallest zeros of sequences of polynomials known as the ``birth and death process orthogonal polynomials'' (see \cite{MouradB}).  

Let us fix some notations used in this paper. The notation $\PP_n$ will
be used for the set of algebraic polynomials of degree not exceeding
$n$ with real coefficients. A labeled bold capital
letter (e.g., $\mathbf{M}_m$) stands for a symmetric $m\times m$
matrix with real elements, i.e. a Hermitian matrix. The $m\times m$ identity matrix is
denoted by $\E_m$. By $\la_{\min}(\mathbf{M}_m)$ and
$\la_{\max}(\mathbf{M}_m)$ we mean, respectively, the smallest and
the largest eigenvalues of $\mathbf{M}_m$. 
The paper is organized as follows. In Section~2 we exploit some
properties of the Laguerre polynomials to prove that
$c_n^{-1}=\frac{1}{2}\,\la_{\min}(\A_n)$, where $\A_n$ is a
specific Jacobi matrix. In Section~3
we show that $\la_{\min}(\A_n)=\la_{\min}(\B_m)$, where $\B_m$ is a
Jacobi matrix of size $m=\lfloor\frac{n+1}{2}\rfloor$. In Section~4
we prove that $\D_{m}=\B_m-\frac{1}{2}\E_m$ is positive definite and
then obtain two-sided estimates for $\la_{\min}(\D_m)$.
Theorem~\ref{t1.1} is proved in Section~5. The proof of
Theorem~\ref{t1.2} is given in Section~6. The lower bound for
$\,d_n\,$ is shown by constructing an appropriate sequence, while
for the upper bound we utilize the analysis performed in the proof
of the continuous case.

\section{Hardy inequality \eqref{e1.3} and the smallest eigenvalue
of a matrix}

Throughout this section we use the customary notation
$\{L_k^{(\al)}\}_{k\in \mathbb{N}_0}$ for the Laguerre polynomials,
which are the orthogonal polynomials on $(0,\infty)$ with respect to
the Laguerre weight function $w_{\al}(x)=x^{\al}e^{-x}$, $\al>-1$.
For the sake of simplicity, when $\al=0$ we write $L_k$ instead of
$L_k^{(0)}$.

For easier reference, we collect in the following lemma some
well-known properties of the Laguerre polynomials (see \cite[eqns.
(5.1.7), (5.1.12)--(5.1.14)]{Sze75}).
\begin{lemma}\label{l2.1}
The following are properties of the Laguerre polynomials:
\begin{eqnarray*}
(i)~&& L_n^{(\al)}(0)={n+\al\choose n}\,;\\
(ii)~&& L_n^{(\al+1)}(x)=\sum_{\nu=0}^n L_{\nu}^{(\al)}(x)\,;\\
(iii)~&& L_n^{(\al)}(x)=L_n^{(\al+1)}(x)-L_{n-1}^{(\al+1)}(x)\,;\\
(iv)~&& \frac{d}{dx}\,L_n^{(\al)}(x)=-L_{n-1}^{(\al+1)}(x)
=x^{-1}\big(n\,L_n^{(\al)}(x)-(n+\al)\,L_{n-1}^{(\al)}(x)\big)\,;\\
(v)~&& \int_{0}^{\infty}e^{-x}L_n^2(x)\,dx=1\,.
\end{eqnarray*}
\end{lemma}

We proceed with reformulating Hardy inequality \eqref{e1.3} into the
language of linear algebra. Instead of the sharp constant $c_n$ in
\eqref{e1.3}, we shall examine its reciprocal,
\begin{equation}\label{e2.1}
\frac{1}{c_n}=\inf_{f\in\HH_n}\frac{\int_{0}^{\infty}f^2(x)\,dx}
{\int_{0}^{\infty}\big(\frac{1}{x}\int_{0}^{x}f(t)\,dt\big)^2\,dx}\,.
\end{equation}

Let us set
\begin{equation}\label{e2.2}
F(x):=\int\limits_{0}^{x}f(t)\,dt\,,\qquad f\in\HH_n.
\end{equation}
Then $F(x)=e^{-x/2}p_n(x)\,$, where $p_n\in\PP_n$ and $p_n(0)=0$. In
view of Lemma~\ref{l2.1}\,(i), the polynomial $p_n$ can be
represented in the form
$$
p_n(x)=\sum_{k=1}^{n}b_k\,\varphi_k(x),\qquad \text{ where }\ \ \ 
\varphi_k(x):=k\big(L_k(x)-L_{k-1}(x)\big)\,.
$$
Thus,
\begin{equation}\label{e2.3}
F(x)=e^{-x/2}\sum_{k=1}^{n}b_{k}\,k\,\big(L_k(x)-L_{k-1}(x)\big)
\end{equation}
and
$$
f(x)=F^{\prime}(x)=e^{-x/2}\,
\Big(p_n^{\prime}(x)-\frac{1}{2}\,p_n(x)\Big)\,.
$$
Then we have
$$
p^\prime_n(x)=\sum_{k=1}^{n}b_k\,\varphi_k^{\prime}(x)
=\sum_{k=1}^{n}k\,b_k\big(L_k^{\prime}(x)-L_{k-1}^{\prime}(x)\big)
=-\sum_{k=1}^{n}k\,b_k\,L_{k-1}(x)\,,
$$
where, for the last equality we have used Lemma~\ref{l2.1}\,(iv) and
(iii). Hence,
\[
\begin{split}
f(x)&=e^{-x/2}\,\Big(-\sum_{k=1}^{n}k\,b_k\,L_{k-1}(x)
-\frac{1}{2}\,\sum_{k=1}^{n}k\,b_k\,\big(L_k(x)-L_{k-1}(x)\big)\Big)\\
&=-\frac{1}{2}\,e^{-x/2}\sum_{k=1}^{n}k\,b_k\,
\big(L_{k-1}(x)+L_k(x)\big)\,.
\end{split}
\]
Now we apply Lemma~\ref{l2.1}\,(v) to obtain
\begin{equation}\label{e2.4}
\begin{split}
\int\limits_{0}^{\infty}f^2(x)\,dx=&\frac{1}{4}\int\limits_{0}^{\infty}
e^{-x}\Big[\sum_{k=1}^{n}k^2b_k^2\big(L_k^2(x)+L_{k-1}^2(x)\big)\\
&
+2\sum_{k=1}^{n-1}k(k+1)b_kb_{k+1}L_k^2(x)\Big]\,dx\\
=&\frac{1}{2}\Big[\sum_{k=1}^{n}k^2b_k^2+\sum_{k=1}^{n-1}k(k+1)b_kb_{k+1}\Big]\,.
\end{split}
\end{equation}
Next, from \eqref{e2.2}, \eqref{e2.3} and
Lemma~\ref{l2.1}\,(iv),\,(ii) and (v) we deduce
\[
\begin{split}
\int\limits_{0}^{\infty}\Bigg(\frac{1}{x}\int\limits_{0}^{x}f(t)\,dt\Bigg)^2\,dx
&=\int\limits_{0}^{\infty}e^{-x}\,\Bigg[\sum_{k=1}^{n}b_k\,k\,
\frac{L_k(x)-L_{k-1}(x)}{x}\Bigg]^2\,dx\\
&=\int\limits_{0}^{\infty}e^{-x}
\Bigg[-\sum_{k=1}^{n}b_k\,L_{k-1}^{(1)}(x)\Bigg]^2\,dx\\
&=\int\limits_{0}^{\infty}e^{-x}
\Bigg[\sum_{k=1}^{n}b_k\,\sum_{j=0}^{k-1}L_j(x)\Bigg]^2\,dx\\
&=\int\limits_{0}^{\infty}e^{-x}
\Bigg[\sum_{j=0}^{n-1}\Bigg(\sum_{k=j+1}^{n}b_{k}\Bigg)L_j(x)\Bigg]^2\,dx\\
&=\sum_{j=0}^{n-1}\Bigg(\sum_{k=j+1}^{n}b_{k}\Bigg)^2\,.
\end{split}
\]
The latter representation, after the substitution
$$
b_k=t_k-t_{k+1},\quad k=1,\ldots,n\,,\quad t_{n+1}:=0\,,
$$
simplifies to
\begin{equation}\label{e2.5}
\int\limits_{0}^{\infty}\Bigg(\frac{1}{x}\int\limits_{0}^{x}f(t)\,dt\Bigg)^2\,dx
=\sum_{k=1}^{n}t_k^2\,,
\end{equation}
while \eqref{e2.4} becomes
$$
\int\limits_{0}^{\infty}f^2(x)\,dx=
\frac{1}{2}\Big[\sum_{k=1}^{n}k^2\big(t_k-t_{k+1}\big)^2
+\sum_{k=1}^{n-1}k(k+1)\big(t_k-t_{k+1}\big)\big(t_{k+1}-t_{k+2}\big)\Big].
$$
After evaluation of the coefficients of the quadratic form in the
right-hand side we find
\begin{equation}\label{e2.6}
\int\limits_{0}^{\infty}f^2(x)\,dx=\frac{1}{2}\, \mathbf{t}^{\intercal} \A_n \mathbf{t}\,,
\end{equation}
where $\mathbf{t}=(t_1,t_2,\ldots,t_n)^{\intercal}$ and $\A_n=\{a_{ij}\}$
is the $n\times n$ Jacobi matrix with elements
\begin{equation}\label{e2.7}
\begin{array}{lll}
&a_{kk}=k^2-k+1, & 1\leq k\leq n\,,\\
&a_{k,k+2}=a_{k+2,k}=-k(k+1),& 1\leq k\leq n-2\,,\\
&a_{jk}=0,  & |j-k|\ne 0,\,2, \quad  1\leq j,\,k\leq n\,.
\end{array}
\end{equation}
From \eqref{e2.5}, \eqref{e2.6} and the Rayleigh-Ritz theorem we
conclude that
$$
\inf_{f\in\HH_n}\frac{\int\limits_{0}^{\infty}f^2(x)\,dx}
{\int\limits_{0}^{\infty}
\big(\frac{1}{x}\int\limits_{0}^{x}f(t)\,dt\big)^2\,dx}
=\frac{1}{2}\inf_{\mathbf{t}\in
\mathbb{R}^n}\frac{\mathbf{t}^{\intercal}\A_n\mathbf{t}}
{\mathbf{t}^{\intercal}\mathbf{t}}=\frac{1}{2}\,\la_{\min}(\A_n)\,.
$$
Comparison of this result with \eqref{e2.1} shows that we have
proved the following statement:
\begin{proposition}\label{p2.2}
The best constant $c_n$ in the Hardy inequality \eqref{e1.3}
satisfies
$$
c_n^{-1}=\frac{1}{2}\,\la_{\min}(\A_n)\,,
$$
where $\la_{\min}(\A_n)$ is the smallest eigenvalue of the $n\times
n$ matrix $\A_n$ with elements given by \eqref{e2.7}.
\end{proposition}

\section{$\la_{\min}(\A_n)$ as the smallest eigenvalue of a Jacobi matrix $\B_m$}

In view of Proposition~\ref{p2.2}, we need to find (or estimate) the
smallest eigenvalue of $\A_n$. By rearrangement of the rows and
columns of $\A_n$ one readily realizes that the characteristic
polynomial of $\A_n$, $P(\la)=|\la\E_n-\A_n|$ is represented as a
product of the characteristic polynomials of two Jacobi matrices.
Namely, we have
$$
P(\la)=|\la\E_n-\A_n|=\Big|\la\E_{\lfloor\frac{n+1}{2}\rfloor}-
\B_{\lfloor\frac{n+1}{2}\rfloor}\Big|.\Big|\la\E_{\lfloor\frac{n}{2}\rfloor}-
\C_{\lfloor\frac{n}{2}\rfloor}\Big|\,,
$$
where the $m\times m$ matrix $\B_m=\{b_{ij}\}$ has elements
\begin{equation}\label{e3.1}
\begin{array}{lll}
&b_{kk}=4k^2-6k+3, & 1\leq k\leq m\,,\\
&b_{k,k+1}=b_{k+1,k}=-k(2k-1),& 1\leq k\leq m-1\,,\\
&b_{jk}=0,  & |j-k|>1,\quad  1\leq j,\,k\leq m\,,
\end{array}
\end{equation}
and the $m\times m$ matrix $\C_m=\{c_{ij}\}$ has elements
\begin{equation}\label{e3.2}
\begin{array}{lll}
&c_{kk}=4k^2-2k+1, & 1\leq k\leq m\,,\\
&c_{k,k+1}=c_{k+1,k}=-k(2k+1),& 1\leq k\leq m-1\,,\\
&c_{jk}=0,  & |j-k|>1,\quad  1\leq j,\,k\leq m\,.
\end{array}
\end{equation}

Now we prove the following statement:
\begin{proposition}\label{p3.1}
There holds
$$
\la_{\min}(\A_n)=\la_{\min}\big(\B_{\lfloor\frac{n+1}{2}\rfloor}\big)\,.
$$
\end{proposition}

\begin{proof}
We consider separately two cases.

\emph{Case 1: $n=2m$}. It follows from \eqref{e3.1} and \eqref{e3.2}
that $\C_m-\B_m=2\G_m$, where $\G_m=\{g_{ij}\}$ is a symmetric
tri-diagonal $m\times m$ matrix with non-zero elements
$g_{kk}=2k-1$, $k=1,\ldots,m$ and $g_{k,k+1}=g_{k+1,k}=-k$,
$\,k=1,\ldots, m-1$\,. We shall show by induction that $|\G_k|=k!$
for every $k\in \mathbb{N}$. Indeed, we have $|\G_1|=1$ and
$|\G_2|=2$, and the induction passage from $k$ to $k+1$ easily
follows from the recurrence relation
$|\G_{k+1}|=(2k+1)\,|\G_k|-k^2\,|\G_{k-1}|$, $\,k\geq 2$. By
Sylvester's criterion we conclude that the matrix $\G_m$ is positive
definite, i.e., for every vector
$\mathbf{x}=(x_1,x_2,\ldots,x_m)^{\intercal}\in \mathbb{R}^m$,
$\mathbf{x}\ne \mathbf{0}$, there holds $\mathbf{x}^{\intercal}
\G_m\mathbf{x}>0$, or, equivalently, $\mathbf{x}^{\intercal}
\mathbf{C}_m\mathbf{x}>\mathbf{x}^{\intercal} \mathbf{B}_m\mathbf{x}$. By
the Rayleigh--Ritz theorem we conclude that
$$
\frac{\mathbf{x}^{\intercal}
\C_m\mathbf{x}}{\mathbf{x}^{\intercal}\mathbf{x}}>
\frac{\mathbf{x}^{\intercal}
\B_m\mathbf{x}}{\mathbf{x}^{\intercal}\mathbf{x}}\geq \inf_{\mathbf{z}\in
\mathbb{R}^m}\frac{\mathbf{z}^{\intercal}
\B_m\mathbf{z}}{\mathbf{z}^{\intercal}\mathbf{z}}= \la_{\min}(\B_m)\,.
$$
Hence,
\begin{equation}\label{e3.3}
\la_{\min}(\C_m)>\la_{\min}(\B_m),
\end{equation}
and consequently
$$
\la_{\min}(\A_n)=\min\big\{\la_{\min}(\B_m),\la_{\min}(\C_m)\big\}
=\la_{\min}(\B_m)\,.
$$

\emph{Case 2: $n=2m+1$}. Since $\B_m$ and $\B_{m+1}$ are embedded
Jacobi matrices, their eigenvalues are zeros of consecutive
orthogonal polynomials, which therefore interlace. Hence,
$$
\la_{\min}(\B_{m+1})<\la_{\min}(\B_m)<\la_{\min}(\C_m)\,,
$$
where the last inequality follows from \emph{Case 1}. Therefore,
$$
\la_{\min}(\A_n)=\min\big\{\la_{\min}(\B_{m+1}),\la_{\min}(\C_m)\big\}
=\la_{\min}(\B_{m+1})\,.
$$
The proof of Proposition~\ref{p3.1} is complete.
\end{proof}

\section{Two-sided estimates for $\la_{\min}(\B_m)$}

In view of Propositions~\ref{p2.2} and \ref{p3.1}, we need estimates
for the smallest eigenvalue $\la_{\min}(\B_m)$ of the matrix $\B_m$. In fact, first we prove in Proposition \ref{p4.1} that $\D_m=\B_m-\frac{1}{2}\,\E_m$ is positive definite and then 
estimate the smallest eigenvalue $\la_{\min}(\D_m)$ of the matrix 
$\D_m$. In order to do this we adopt a relatively straightforward strategy which, however requires to overcome some technical difficulties.  

 Since $\D_m$ are Jacobi matrices, as it is well known (see \cite{MouradB}), we may associate with them a sequence 
of algebraic polynomials $\{d_m\}$ which are orthonormal with respect to a Borel measure, such that the eigenvalues 
of $\D_m$ coincide with the zeros of $d_m(x)$. Then we perform two renormalisations of these polynomials obtaining the nearly monic ones 
which are denoted by $P_m$ and are normalised in such a way that the leading coefficient of $P_m$ is $(-1)^m$. Finally we normalise them once again to 
obtain the sequence of polynomials $Q_m$ which are orthogonal in (a subset of) $(0,\infty)$ and are normalised by  $Q_m(0)=1$. Any sequence of orthogonal polynomials with 
such an orthogonal property and a normalisation are called the ``birth and death process polynomials''. Some results concerning zeros of these families of polynomials were obtained by 
M. E. H. Ismail \cite{Mourad87, Mourad89}.  However, since we need sharp estimates only for the smallest zeros of $Q_m(\lambda)$, we use the fact that it is larger than the point of 
intersection of the line, tangent to the graph of $Q_m(\lambda)$ at $\lambda=0$, and the real line. A bit more sophisticated observation is used to obtain the upper bound for the smallest zero
of  $Q_m(\lambda)$.     
\begin{proposition}\label{p4.1}
The matrix $\D_m=\B_m-\frac{1}{2}\,\E_m$ is positive definite.
\end{proposition}

\begin{proof}
From \eqref{e3.1} we find that the elements $d_{ij}$ of $\D_m$ are
given by
\begin{equation}\label{e4.1}
\begin{array}{lll}
&d_{kk}=4k^2-6k+\frac{5}{2}, & 1\leq k\leq m\,,\\
&d_{k,k+1}=d_{k+1,k}=-k(2k-1),& 1\leq k\leq m-1\,,\\
&d_{jk}=0,  & |j-k|>1,\quad 1\leq j,\,k\leq m\,.
\end{array}
\end{equation}
On using \eqref{e4.1}, one readily obtains the recurrence relation
\begin{equation}\label{e4.2}
|\D_m|=\Big(4m^2-6m+\frac{5}{2}\Big)\,|\D_{m-1}|-(m-1)^2(2m-3)^2\,|\D_{m-2}|\,,
\qquad m\geq 3\,.
\end{equation}

We shall show that
\begin{equation} \label{e4.3}
|\D_m| = \frac{[(2m-1)!!]^2}{2^m},\qquad m\in \mathbb{N},
\end{equation}
whence the claim of Proposition~\ref{p4.1} would follow from
Sylvester's criterion. The proof of \eqref{e4.3} is by induction
with respect to $m$. It is easily verified that \eqref{e4.3} is true
for $m=1,\,2$. Assume that \eqref{e4.3} is true for all natural
numbers $1,2,\ldots,m-1$. Then, using \eqref{e4.2} and the
inductional hypothesis, we accomplish the induction step as follows:
\begin{align*}
|\D_m| & = \Big(4m^2-6m+\frac{5}{2}\Big)\frac{[(2m-3)!!]^2}{2^{m-1}}
- (m-1)^2(2m-3)^2\,\frac{[(2m-5)!!]^2}{2^{m-2}} \\
& = \frac{[(2m-3)!!]^2}{2^m}\Big(2\Big(4m^2-6m+\frac52\Big)-4(m-1)^2\Big) \\
& = \frac{[(2m-1)!!]^2}{2^m}\,.
\end{align*}
The proof of Proposition~\ref{p4.1} is complete.
\end{proof}

\begin{corollary}\label{c4.2}
For every $m\in\mathbb{N}$, there holds
$$
\la_{\min}(\B_m)=\frac{1}{2}+\la_{\min}(\D_m)>\frac{1}{2}\,.
$$
\end{corollary}

\subsection{Estimating $\la_{\min}(\D_m)$ from below}

Here we prove the following statement.
\begin{proposition}\label{p4.3}
For every $m\in \mathbb{N}$ the smallest eigenvalue
$\la_{\min}(\D_m)$ of the matrix $\D_m$ with elements given by
\eqref{e4.1} satisfies
$$
\la_{\min}(\D_m)>\frac{4}{\ln^2 m+8\ln m+8}\,.
$$
\end{proposition}

We need some preparation before proving Proposition~\ref{p4.3}. Let
us consider the characteristic polynomial of $\D_m$,
$P_m(\la)=|\D_m-\la\,\E_m|$. From \eqref{e4.1} one readily obtains
the recurrence relation
\begin{equation}\label{e4.4}
\begin{split}
&P_m(\la)=\Big(4m^2-6m+\frac{5}{2}-\la\Big)\,P_{m-1}(\la)
-(m-1)^2(2m-3)^2P_{m-2}(\la)\,, \\
&P_0(\la):=1,\ \ P_1(\la)=\frac{1}{2}-\la
\end{split}
\end{equation}
(note that \eqref{e4.2} corresponds to \eqref{e4.4} in the case
$\la=0$).

Let us set
$$
P_m(\la) = \frac{[(2m-1)!!]^2}{2^m}\,Q_m(\la),\qquad m\in
\mathbb{N}_0.
$$
Then, in view of \eqref{e4.3}, $Q_m(\la)$ takes the form
\begin{equation}\label{e4.5}
Q_m(\la)=1+\sum_{j=1}^{m}(-1)^{j}q_{jm}\,\la^{j}, \qquad m\in
\mathbb{N}_0,
\end{equation}
and \eqref{e4.4} becomes
\begin{equation}\label{e4.6}
\begin{split}
&(2m-1)^2Q_{m}(\la)=\big(8m^2-12m+5-2\la\big)\,Q_{m-1}(\la)
-4(m-1)^2\,Q_{m-2}(\la)\,,\\
&Q_0(\la):=1,\ \ Q_1(\la)=1-2\la\,.
\end{split}
\end{equation}

Our next goal is to obtain two-sided estimates for the coefficients
$q_{1m}$, $m\in \mathbb{N}$. From \eqref{e4.6} we have
$$
q_{10}=0,\ \ q_{11}=2\,,
$$
and
$$
(2m-1)^2q_{1m}=(8m^2-12m+5)\,q_{1,m-1}-4(m-1)^2q_{1,m-2}+2\,,\quad
m\geq 2\,.
$$
Let us set
$$
y_{k}:=q_{1k}-q_{1,k-1}\,,\qquad k\in \mathbb{N},
$$
then the sequence $\{y_{k}\}_{k\in \mathbb{N}}$ is determined by
\begin{equation}\label{e4.7}
(2k-1)^2y_{k}=4(k-1)^2y_{k-1}+2\,,\qquad y_1=2\,.
\end{equation}

\begin{lemma}\label{l4.4}
Let $\{y_{k}\}_{k\in \mathbb{N}}$ be defined by \eqref{e4.7}. Then
the following inequalities hold true:
$$
\frac{\ln k }{2k}<y_k\leq\frac{\ln k +4}{2k}\,,\qquad k\in
\mathbb{N}\,.
$$
\end{lemma}

\begin{proof} \emph{a) The upper estimate}. The proof goes by
induction, and the base case $k=1$ is obvious. Assuming the upper
estimate is true for some $k\in \mathbb{N}$, we obtain
$$
y_{k+1}=\frac{4k^2y_k+2}{(2k+1)^2}\leq\frac{2k(\ln k+4)+2}{(2k+1)^2} {\le} \frac{\ln(k+1)+4}{2(k+1)}\,.
$$
The last inequality is equivalent to
$$
4k(k+1)\ln\Big(1+\frac{1}{k}\Big)+\ln(k+1)\geq 4k\,,\qquad k\in
\mathbb{N}\,.
$$
This inequality is verified to be true, using $\ln
(1+x)>x-x^2/2\,,\ x>0$, as follows:
$$
4k(k+1)\ln\Big(1+\frac{1}{k}\Big)+\ln(k+1)>
4k(k+1)\Big(\frac{1}{k}-\frac{1}{2k^2}\Big)=4k+\frac{2(k-1)}{k} \geq
4k\,.
$$

\emph{b) The lower estimate}. We shall prove by induction the
(slightly stronger) inequality
\begin{equation}\label{e4.8}
y_{k}>\frac{\ln k}{2k-1}\,,\qquad k\in \mathbb{N}\,.
\end{equation}
Inequality \eqref{e4.8} is easily verified to be true for $k=1,\,2$.
Assuming that \eqref{e4.8} holds true for $k-1$, where $k\geq 3$, we
shall prove it for $k$. In order to do so, we observe that the
function
$f(x)=\ln\big(1-\frac{1}{x}\big)+\frac{1}{x-1}-\frac{1}{4(x-1)^2}$
is monotone decreasing for $x\in(2,\infty)$, since
$f'(x)=\frac{2-x}{2x(x-1)^3}$. Then, from $\lim_{x\to +\infty}
f(x)=0$ it follows that $f(x)>0$, hence
\begin{equation}\label{e9}
\ln(x-1)>\ln x -\frac{1}{x-1}+\frac{1}{4(x-1)^2}, \qquad
x\in[2,\infty).
\end{equation}

Applying the induction hypothesis $y_{k-1}>\frac{\ln(k-1)}{2k-3}$
and \eqref{e9} we obtain
\begin{align*}
  y_k & = \frac{4(k-1)^2 y_{k-1}+2}{(2k-1)^2}
        > \frac{4(k-1)^2\frac{\ln(k-1)}{2k-3}+2}{(2k-1)^2} \\
      & = \frac{4(k-1)^2\ln(k-1)+4k-6}{(2k-3)(2k-1)^2}
        > \frac{4(k-1)^2\big[\ln k -\frac{1}{k-1} +\frac{1}{4(k-1)^2}\big]+4k-6}{(2k-3)(2k-1)^2} \\
      & = \frac{4(k-1)^2\ln k - 1}{(2k-3)(2k-1)^2}
        = \frac{1}{2k-1}\Big(\ln k +\frac{\ln k - 1}{(2k-1)(2k-3)}\Big) \\
      & > \frac{\ln k}{2k-1}.
\end{align*}

Thus, \eqref{e4.8} is verified and Lemma~\ref{l4.4} is proved.
\end{proof}

\noindent \textit{Proof of Proposition~\ref{p4.3}.} In fact, we make
use only of the upper estimate for $y_k$ in Lemma~\ref{l4.4}, the
lower one is just to show that the major term is correctly
identified. Let us recall that $y_k=q_{1k}-q_{1,k-1}$ with
$q_{1,0}=0$. Using the upper bound in Lemma~\ref{l4.4} and the fact
that $h(x)=\frac{\ln x+4}{x}$ is a decreasing function in
$[1,\infty)$, we find
\[
\begin{split}
q_{1m}&=\sum_{k=1}^{m}y_k\leq \sum_{k=1}^{m}\frac{\ln k +4}{2k}
=2+\sum_{k=2}^{m}\frac{\ln k +4}{2k}\\
&\leq 2+\frac{1}{2}\,\int\limits_{1}^{m}h(x)\,dx=\frac{\ln^2 m+8\ln
m+8}{4}\,,
\end{split}
\]
hence
\begin{equation}\label{e4.10}
q_{1m}\leq\frac{\ln^2 m+8\ln m+8}{4}\,.
\end{equation}
The polynomial $Q_m(\la)$ in \eqref{e4.5} is the characteristic
polynomial of the positive definite matrix $\D_m$. Then its
reciprocal polynomial
$$
R_{m}(\la)=\la^{m}\,Q_m(\la^{-1})=\la^{m}+
\sum_{j=1}^{m}(-1)^{j}q_{jm}\,\la^{m-j}
$$
is the monic characteristic polynomial of the matrix $\D_m^{-1}$,
which is also positive definite. Thus, all the zeros of $R_m$ are
positive, and their sum equals $q_{1m}$. Therefore, by
\eqref{e4.10},
$$
\frac{1}{\la_{\min}(\D_m)}=\la_{\max}(\D_m^{-1})<q_{1m} \le
\frac{\ln^2 m+8\ln m+8}{4}\,,
$$
and hence
$$
\la_{\min}(\D_m)>\frac{4}{\ln^2 m+8\ln m+8}\,.
$$
Proposition~\ref{p4.3} is proved. \qed

\subsection{Estimate of $\la_{\min}(\D_m)$ from above}

Denote the entries of $\D_m^{-1}$ by $\al_{ij}$, $\;1\leq i,\,j\leq
m$. $\D_m^{-1}$ is a Hermitian matrix, therefore
$\la_{\max}(\D_m^{-1})\geq \al_{ii}$ for $i=1,\ldots,m$, and
$\al_{ii}>0$, since $\D_m^{-1}$ is positive definite. In particular,
we have
\begin{equation}\label{e4.11}
\frac{1}{\la_{\min}(\D_m)}=\la_{\max}(\D_m^{-1})\geq
\al_{11}=\frac{\d_{m-1}}{|\D_m|}\,,
\end{equation}
where $\d_{m-1}$ is the determinant of the $(m-1)\times(m-1)$
matrix, obtained from $\D_m$ by deletion of the first row and column
in $\D_m$. We observe that the sequence $\{\d_k\}_{k\in \mathbb{N}}$
satisfies the same (modulo a shift of the indices) recurrence
relation as $\{|D_k|\}$, namely,
\begin{equation}\label{e4.12}
\begin{split}
&\d_{k-1}=\Big(4k^2-6k+\frac{5}{2}\Big)\,\d_{k-2}
-(k-1)^2(2k-3)^2\,\d_{k-3}\,, \qquad k\geq 4\,,\\
&\d_1=\frac{13}{2}\,,\ \ \d_2=\frac{389}{4}\,.
\end{split}
\end{equation}
We set
$$
v_k:=\d_k-\frac{(2k+1)^2}{2}\,\d_{k-1}\,,
$$
then \eqref{e4.12} simplifies to the recurrence equation
$$
v_{k}=2k^2\,v_{k-1}\,,\qquad k\geq 3\,,\qquad v_2=16\,,
$$
whose solution is $\,v_k=2^k\big(k!\big)^2$. Now \eqref{e4.12}
becomes
\begin{equation}\label{e4.13}
\d_{k}-\frac{(2k+1)^2}{2}\,\d_{k-1}=2^k\big(k!\big)^2\,, \qquad
k\geq 2\,,\qquad \d_1=\frac{13}{2}\,.
\end{equation}
The substitution
$$
\d_k=\frac{\big[(2k+1)!!\big]^2}{2^{k+1}}u_k\,,
$$
which, in view of \eqref{e4.3}, can be written as
\begin{equation}\label{e4.14}
u_k=\frac{\d_k}{|\D_{k+1}|},
\end{equation}
transforms \eqref{e4.13} into the recurrence equation
\begin{equation}\label{e4.15}
u_{k}-u_{k-1}=2\,\Bigg(\frac{(2k)!!}{(2k+1)!!}\Bigg)^2\,, \qquad
k\geq 2\,,\qquad u_1=\frac{26}{9}\,.
\end{equation}
From \eqref{e4.15} we find
$$
u_{m-1}=u_1+\sum_{k=2}^{m-1}\big(u_k-u_{k-1}\big)=
\frac{26}{9}+2\,\sum_{k=2}^{m-1}\Bigg(\frac{(2k)!!}{(2k+1)!!}\Bigg)^2
=2\,\sum_{k=0}^{m-1}\Bigg(\frac{(2k)!!}{(2k+1)!!}\Bigg)^2\,.
$$
By multiplying inequalities $(2j)^2>(2j-1)(2j+1)$, $\,1\leq j\leq
k$, we obtain
$$
\Bigg(\frac{(2k)!!}{(2k+1)!!}\Bigg)^2>\frac{1}{2k+1}\,,
$$
therefore
$$
u_{m-1}>2\,\sum_{k=0}^{m-1}\frac{1}{2k+1}
>\int\limits_{0}^{m}\frac{dx}{x+\frac{1}{2}}
>\ln\Big(m+\frac{1}{2}\Big)\,.
$$
From \eqref{e4.14} we infer
$$
\frac{\d_{m-1}}{|\D_m|}=u_{m-1}>\ln\Big(m+\frac{1}{2}\Big)\,,
$$
which, in view of \eqref{e4.11}, completes the proof of
\begin{proposition}\label{p4.5}
For every $m\in \mathbb{N}$, there holds
$$
\la_{\min}(\D_m)<\frac{1}{\ln\big(m+\frac{1}{2}\big)}\,.
$$
\end{proposition}

\section{Proof of Theorem~\ref{t1.1}}

With $m=\Big\lfloor\frac{n+1}{2}\Big\rfloor$, the consecutive application of 
Proposition~\ref{p2.2}, Proposition~\ref{p3.1} and Corollary~\ref{c4.2} yields 
following relations for the best constant $c_n$  in the Hardy inequality
\eqref{e1.3}:
\[
\frac{1}{c_n}=\frac{1}{2} \la_{\min}(\A_n) = \frac{1}{2}\la_{\min}(\B_m)
=\frac{1}{2}\Big(\frac{1}{2}+\la_{\min}(\D_m)\Big) =\frac{2\la_{\min}(\D_m)+1}{4}\,.
\]
Thus,
$$
c_n=\frac{4}{2\la_{\min}(\D_m)+1}=
4\Big(1-\frac{2\la_{\min}(\D_m)}{1+2\la_{\min}(\D_m)}\Big)\,.
$$
Furthermore, if $0<\underline{\la}<\la_{\min}(\D_m)<\overline{\la}$,
then
\begin{equation}\label{e5.1}
4\Big(1-\frac{2\overline{\la}}{1+2\overline{\la}}\Big)<c_n<
4\Big(1-\frac{2\underline{\la}}{1+2\underline{\la}}\Big)\,.
\end{equation}
According to Proposition~\ref{p4.3}, we have
$$
\la_{\min}(\D_m)>\frac{4}{\ln^2 m+8\ln m+8} \geq
\frac{4}{\ln^2\frac{n+1}{2}+8\ln\frac{n+1}{2}+8}=:\underline{\la}\,,
$$
while Proposition~\ref{p4.5} implies
$$
\la_{\min}(\D_m)<\frac{1}{\ln\big(m+\frac{1}{2}\big)}\leq
\frac{1}{\ln\frac{n+1}{2}}=:\overline{\la}\,.
$$
By substituting these bounds for $\la_{\min}(\D_m)$ in \eqref{e5.1},
we obtain the claim of Theorem~\ref{t1.1}\,.

\section{Proof of Theorem~\ref{t1.2}}
The proof of the lower bound for $\,d_n\,$ follows from the specific
choice of the sequence
$$
a_j=\sqrt{j}-\sqrt{j-1}\,,\qquad j=1,\ldots,n\,.
$$
With this choice, we have
$$
\sum_{k=1}^{n}\Big(\frac{1}{k}\sum_{j=1}^{k}a_j\Big)^2
=\sum_{k=1}^{n}\frac{1}{k}=H_n\,,
$$
where $\,H_n\,$ is the usual notation for the $n$-th harmonic
number. On the other hand,
$$
\sum_{k=1}^{n}a_k^2=1+\sum_{k=2}^{n}\frac{1}{(\sqrt{k}+\sqrt{k-1})^2}
<1+\frac{1}{4}\sum_{k=1}^{n}\frac{1}{k}=1+\frac{H_{n}}{4}\,.
$$
Now, using the simple inequality $\,H_{n}>\ln n$, we conclude that
$$
d_n > \frac{H_{n}}{1+H_{n}/4}
=4-\frac{16}{H_{n}+4}>4-\frac{16}{\ln n + 4}\,,
$$
which is the desired lower bound for $d_n$.

For the proof of our upper bound for $\,d_n$, we perform the
change of variables
$$
\sum_{j=1}^{k}a_j=k\,b_k\,,\qquad k=1,\ldots, n.
$$
Then inequality \eqref{e1.5} becomes
$$
\sum_{k=1}^{n}b_k^2\leq
d_n\sum_{k=1}^{n}\big(k\,b_k-(k-1)b_{k-1}\big)^2\,.
$$
The smallest $\,d_n$ for which the latter holds is given by
\begin{equation}\label{e6.1}
\frac{1}{d_n}=\min_{\mathbf{b} \in \mathbb{R}^n}
\frac{\mathbf{b}^{\intercal}\mathbf{F}_n\mathbf{b}}
{\mathbf{b}^{\intercal}\mathbf{b}}=\la_{\min}(\mathbf{F}_n)\,,
\end{equation}
where $\,\mathbf{F}_n\,$ is an $n\times n$ symmetrical Jacobi matrix
whose diagonal entries are
$$
f_{kk}=2k^2\ \ \mathrm{for}\ \ 1\leq k\leq n-1,\ \ \mathrm{and}\ \ f_{nn}=n^2\,,
$$
and the only non-zero off-diagonal entries are
$$
f_{k,k+1}=f_{k+1,k}=-k(k+1)\,,\quad k=1,\ldots,n-1\,.
$$
It is readily verified that $\mathbf{F}_n$ admits the decomposition
$$
\mathbf{F}_n=\mathbf{U}_n\mathbf{U}_n^{\intercal}\,,
$$
where $\mathbf{U_n}$ is the bi-diagonal matrix whose only nonzero entries are $u_{kk}=k$, for $k=1,\ldots, n$, and $u_{k,k+1}=-k$, for $k=1,\ldots, n-1$. 

The matrices $\,\mathbf{H}_n=\mathbf{U}^{\intercal}_n\mathbf{U}_n\,$ and
$\,\mathbf{F}_n=\mathbf{U}_n\mathbf{U}^{\intercal}_n\,$ are similar, so that they have the same eigenvalues, and in particular,
\begin{equation}\label{e6.2}
\la_{\min}(\mathbf{F}_n)=\la_{\min}(\mathbf{H}_n)\,.
\end{equation}
An easy calculation shows that $\,\mathbf{H}_n=(h_{ij})_{n\times
n}\,$ is a Jacobi matrix with entries
\begin{equation}\label{e6.3}
\begin{array}{lll}
&h_{kk}=2k^2-2k+1, & 1\leq k\leq n\,,\\
&h_{k,k+1}=h_{k+1,k}=-k^2,& 1\leq k\leq n-1\,,\\
&h_{jk}=0,  & |j-k|>1,\quad  1\leq j,\,k\leq n\,.
\end{array}
\end{equation}
Comparison of \eqref{e3.1}--\eqref{e3.2} (with $\,m=n$) and
\eqref{e6.3} shows that
$$
\mathbf{H}_n=\frac{1}{4}\,\big(\mathbf{B}_n+\mathbf{C}_n\big)\,.
$$
From the Rayleigh--Ritz theorem and \eqref{e3.3} we infer
$$
\la_{\min}(\mathbf{H}_n)=\frac{1}{4}\la_{\min}\big(\mathbf{B}_n+\mathbf{C}_n\big)
\geq
\frac{1}{4}\,\big(\la_{\min}(\mathbf{B}_n)+\la_{\min}(\mathbf{C}_n)\big)
>\frac{1}{2}\,\la_{\min}(\mathbf{B}_n)\,.
$$
Corollary~\ref{c4.2} and Proposition~\ref{p4.3}, applied with $m=n$,
yield further
$$
\la_{\min}(\mathbf{H}_n)>\frac{1}{2}\,\la_{\min}(\mathbf{B}_n)
>\frac{1}{4}+\frac{2}{\ln^2 n+8\ln n+8} =\frac{\ln^2n+8\ln n+16}{4(\ln^2n+8\ln n+8)}\,.
$$
Now the latter inequality, \eqref{e6.1} and \eqref{e6.2} imply the
desired upper bound for $\,d_n\,$:
$$
d_n=\frac{1}{\la_{\min}(\mathbf{F}_n)}=\frac{1}{\la_{\min}(\mathbf{H}_n)}<
\frac{4(\ln^2n+8\ln n+8)}{\ln^2n+8\ln n+16}=4\Bigg(1-\frac{8}{(\ln 
n+4)^2}\Bigg)\,.
$$

\bibliographystyle{amsplain}


\end{document}